\documentclass[12pt]{amsart}

\usepackage{fullpage}
\usepackage{enumitem}
\usepackage{scrextend}
\usepackage{setspace}

\usepackage{amsmath}
\usepackage{amsfonts}
\usepackage{amssymb}
\usepackage{amsthm}
\usepackage{mathtools} 

\usepackage{multicol}
\usepackage{graphicx}
\usepackage{accents}
\usepackage{enumitem}
\usepackage{xcolor}
\usepackage{tikz-cd}

\newlength{\bibitemsep}
\newlength{\bibparskip}
\let\oldthebibliography\thebibliography
\renewcommand\thebibliography[1]{%
  \oldthebibliography{#1}%
  \setlength{\parskip}{\bibitemsep}%
  \setlength{\itemsep}{\bibparskip}%
}

\theoremstyle{definition}
\newtheorem{example}{Example}[section]
\newtheorem{definition}{Definition}[section]
\theoremstyle{remark}

\newtheorem{remark}{Remark}[section]

\usepackage{thmtools}
\usepackage{thm-restate}

\declaretheorem[name=Theorem,numberwithin=section, style=definition]{thm} 
\declaretheorem[name=Lemma,sharenumber=thm]{lemma}
\declaretheorem[name=Proposition,sharenumber=thm]{prop}

\newcommand{\inv}{^{-1}}

\begin{document}

\title{Completely bounded subcontexts of a Morita context of unital $C^*$-algebras}
\author{Kathryn McCormick}
\address{Dept of Math and Stats, CSU Long Beach, Long Beach, CA 98040, USA}
\email{kathryn.mccormick@csulb.edu}


\begin{abstract} In this paper, we answer a question of Blecher-Muhly-Paulsen pertaining to identifying topological invariants for completely bounded Morita equivalences of holomorphic cross-section algebras. Given a certain natural subcontext of a strong Morita context of $n$-homogeneous $C^*$-algebras whose spectrum $T$ is an annulus, Blecher-Muhly-Paulsen are able to estimate the norm of a lifting of the identity of a holomorphic subalgebra by a conformal invariant of the annulus and a property of the associated matrix bundle. We give a generalization of the above example in which $T$ is a bordered Riemann surface. While constructing this generalization, we develop a sufficient criterion for when a unital completely bounded Morita equivalence can be factored into a similarity and a strong Morita equivalence.
    \end{abstract}

\maketitle

\section{Introduction}
Suppose that $A$ and $B$ are $C^*$-algebras that are strongly Morita equivalent in the sense of Rieffel \cite{R1974a,Rieffel1976}. A strong Morita equivalence between $A$ and $B$ allows one to travel from a left Hilbert $A$-module $M$ to a left Hilbert $B$-module $N$ (and back) by tensoring with the appropriate Hilbert module. Strong Morita equivalence preserves many relevant properties of $C^*$-algebras \cite{AnHuef2007}, and is used as a tool in many applications, such as to transformation group $C^*$-algebras \cite{Rieffel1982,LeschMoscovici2016}.

The authors of \cite{Blecher2000}, motivated by problems in the representation theory of nonselfadjoint operator algebras, developed two types of Morita equivalence for operator algebras: \emph{strong} Morita equivalence, and the parallel notion of \emph{completely bounded} ($cb$) Morita equivalence \cite[Def.~3.1]{Blecher2000}. If the operator algebras in question are unital $C^*$-algebras, one can relate the two types of Morita equivalence by decomposing a $cb$ Morita equivalence into a strong Morita equivalence and a $cb$-isomorphism: that is, if $A$ and $B$ are unital $C^*$-algebras and if $A$ is $cb$ Morita equivalent to $B$, then $B$ is $cb$-isomorphic to a $C^*$-algebra $B'$ where $B'$ strongly Morita equivalent to $A$ \cite[Thm.~7.11]{Blecher2000}. Also see \cite{Elef2016} for another notion of Morita equivalence for operator algebras.

We wish to pursue several questions surrounding Morita equivalence for unital, not necessarily self-adjoint, operator algebras. Our first motivating question is, under what conditions can the decomposition above be achieved for a $cb$ Morita equivalence of unital operator algebras that are not $C^*$-algebras? 

For unital $C^*$-algebras, any $cb$ isomorphism is similar to a complete isometric ($ci$) isomorphism, but this is not true for unital operator algebras in general \cite[Thm.~2.6, Sec.~3, resp.]{Clouatre2015}. In \cite{Clouatre2015}, Clou\^{a}tre studies how close a $cb$ isomorphism is to factoring as a complete isometric isomorphism composed with similarities. We would like to pursue a similar line of inquiry, but for $cb$ Morita equivalences. Thus, our second question is: if we can decompose a $cb$ Morita equivalence into a $cb$ isomorphism and a strong Morita equivalence, are there cases in which we can we replace the $cb$ isomorphism with a similarity and $ci$ isomorphism?

In particular, we have been motivated to answer the above two questions in the setting of a \emph{subcontext} of a $C^*$-algebra Morita context. Let $A \subset A_1$ and $B \subset B_1$ be operator algebras containing the unit of their respective ambient $C^*$-algebras. Let $X_1$ be a $A_1$-$B_1$ operator bimodule, let $Y_1$ be a $B_1$-$A_1$ operator bimodule, and let $(\cdot, \cdot)$ and $[\cdot, \cdot]$ be $A_1$ (resp.~$B_1$)-valued pairings with sufficient conditions to produce a strong Morita context of $C^*$-algebras in the sense of \cite{Blecher2000}. If there exists such a Morita context, $A_1$ and $B_1$ are said to be \emph{strongly Morita equivalent}. In addition, suppose that there exist closed operator subspaces $X \subset X_1$, $Y \subset Y_1$, which are $A$-$B$ (resp.~$B$-$A$) submodules so that the restriction of the pairing $(\cdot,\cdot)$ (and $[\cdot,\cdot]$) maps onto $A$ (resp.~$B$) and so that $(A,B,X,Y,(\cdot,\cdot),[\cdot,\cdot])$ is a strong or $cb$ Morita context of \cite{Blecher2000}. In \cite[Def.~5.2(ii)]{Blecher1999}, such a collection $(A,B,X,Y,(\cdot,\cdot),[\cdot,\cdot])$ is called a unital subcontext of a $C^*$-context. 
A unital subcontext of a strong Morita context of $C^*$-algebras need \emph{not} be a strong Morita context in and of itself. We would like to understand better the conditions under which a strong Morita equivalence descends to the subcontext. In this paper, we provide sufficient conditions on the subcontext to guarantee that such a $cb$ Morita equivalence is similar to a strong Morita equivalence. In particular, it is enough to have \textit{compatible symmetric lifts} of $1_A$ and $1_B$ (see Definition \ref{def:CompatibleSymmetric}). Our main theorem is the following:

\begin{restatable}{thm}{mainprop} \label{mainprop} Let $(A,B,X,Y,(\cdot,\cdot),[\cdot,\cdot])$ be a unital $cb$ subcontext of a $cb$ Morita context of $C^*$-algebras $(A_1,B_1,X_1,Y_1,(\cdot,\cdot),[\cdot,\cdot])$. Further, assume that there is an algebraic lift of $1_B$ and an algebraic lift of $1_A$ which are compatible symmetric. Then there is a representation $\rho : B_1 \to B(H)$ and an invertible map $S \in B(H)$ so that $S^{-1} \rho(B)S$ is strongly Morita equivalent to $A$.
	\end{restatable} 
\noindent Furthermore, the norm of the map $S$ may be estimated explicitly, as described in Eqn.~\ref{Snormestimate}.

In Section \ref{exexample}, we apply Theorem \ref{mainprop} to a certain subcontext of a Morita context of two $n$-homogeneous $C^*$-algebras over a Riemann surface. We would like to motivate why this example is of interest to us, and why one would expect a theorem relating topological and bundle properties to Morita contexts, as in \cite[Ex.~8.3]{Blecher2000} and Theorem \ref{NotACIContext}. In short, we are motivated by successes in classifying $n$-homogeneous $C^*$-algebras up to $*$-isomorphism and strong Morita equivalence. The strong Morita equivalence class of an $n$-homogeneous $C^*$-algebra is determined by its Dixmier-Douady class, a topological invariant. Moreover, if we fix $n$ and consider $n$-homogeneous $C^*$-algebras up to $*$-isomorphisms that are identity on the center, the class of an algebra is determined by another finer invariant, its bundle class. One can compare the situation for $n$-homogeneous $C^*$-algebras to the case of stable, separable, continuous trace $C^*$-algebras with paracompact spectrum $T$, which are wholly determined up to spectrum-preserving strong Morita equivalence by their $*$-isomorphism\footnote{That is, $C_0(T)$-linear $*$-isomorphism} class 
\cite[5.58]{Raeburn1998}. Below, we will unpack a few more details of the relationship between the Dixmier-Douady invariant and other properties of $n$-homogeneous $C^*$-algebras. 

A unital commutative $C^*$-algebra with compact Hausdorff spectrum $T$ is $*$-isomorphic to $C(T)$. The next best non-commutative model of a unital $C^*$-algebra is the continuous cross-sections of a topological bundle, where the fibres of the bundle are a noncommutative $C^*$-algebra such as $M_n(\mathbb{C})$ or $K(H)$. In particular, any $n$-homogeneous $C^*$-algebra is $*$-isomorphic to the continuous cross sections of a matrix bundle \cite{Tomiyama1961}. 


More explicitly, any bundle over a paracompact space $T$ is determined up to continuous bundle isomorphism\footnote{Here, we assume the bundle isomorphisms are fixing the points in $T$} by its class in $H^1(T, \mathcal{A}ut(M_n\mathbb{C}))$, where $\mathcal{A}ut(M_n\mathbb{C})$ is the sheaf of continuous $PU_n(\mathbb{C})$-valued functions. By \cite{Tomiyama1961}, any unital $n$-homogenous $C^*$-algebra $A$ with spectrum a compact Hausdorff space $T$ has an associated matrix bundle $[c]$ over $T$ which determines $A$ up to $*$-isomorphism. In particular, $A$ is the collection of continuous cross-sections of $[c]$ and is $*$-isomorphic to $C(T) \otimes M_n(\mathbb{C})$ iff the associated bundle is continuously $PU_n(\mathbb{C})$-trivial.

One can define a sequence of maps of (noncommutative) cohomology sets. Let $\mathcal{S}$ be the sheaf of $\mathbb{T}$-valued functions, and let $T$ be a paracompact topological space. Then we have:

\begin{center}
\begin{tikzcd}
  H^1(T, \mathcal{A}ut(M_n\mathbb{C})) \arrow[r,"\Delta"]  & H^2(T, \mathcal{S})  \arrow[r,"\Delta^2"] &  H^3(T; \mathbb{Z}) \\
  \mbox{[c]}  \arrow[r,mapsto] & \Delta([c])  \arrow[r,mapsto]  & \Delta^2([c])=\delta([c]) 
    \end{tikzcd}
    \end{center}

where $\delta([c])$ is the Dixmier-Douady invariant of the $C^*$-algebra of cross sections of $[c]$. If we have a bundle with fibres $K(H)$, $H$ infinite-dimensional, then $\Delta$ would be a bijection; here it is simply a (well-defined) map \cite[4.83]{Raeburn1998}. By \cite{DixmierDouady} and \cite[5.32, 5.33, 5.36]{Raeburn1998}, we know that $\delta([c])=0$ iff the $C^*$-algebra associated to $[c]$ is strongly Morita equivalent to its center\footnote{Where the equivalence \emph{spectrum-preserving} or \emph{over} $T$, see \cite[Def.~5.6]{Raeburn1998}.}. That is, one can show that $\delta([c])=0$ if and only if the bundle defined by $[c]$ is an adjoint bundle \cite[4.85]{Raeburn1998}, and $[c]$ being an adjoint bundle gives a method of defining an appropriate equivalence bimodule. When we are considering bundles with finite-dimensional fibres over arbitary $T$, there do exist nontrivial principal unitary bundles\footnote{For infinite-dimensional fibres, all unitary bundles are trivial.}. So, even if $\delta([c])=0$, one could still have a bundle $[c]$ which is nontrivial. However, when $T=\overline{R}$, where $\overline{R}$ a smoothly bordered Riemann surface with finitely many boundary components, it \emph{is} true that all continuous unitary vector bundles over $\overline{R}$ are trivial\footnote{It is well known that a complex vector bundle over a circle is trivial. Then, apply the fact that $\overline{R}$ is homeomorphic to a wedge of circles.}. Therefore, we get the relationship:

 \begin{multline} \label{eqn:originaldiagram} \delta([c])=0 \overset{(a)}{\iff} [c] \text{ defines an adjoint bundle} \\ \overset{(b)}{\iff} [c] \text{ is a trivial bundle } \\ \hfill \overset{(c)}{\iff} \text{the }C^*\text{-algebra of continuous cross-sections of }[c] \\ \text{ is $*$-isomorphic to } C(\overline{R}) \otimes M_n(\mathbb{C}) 
   \end{multline} 

These four equivalent statements are automatically and trivially satisfied over $\overline{R}$.  However, if we want to refine Equation \ref{eqn:originaldiagram} to subalgebras of holomorphic cross-sections, and subcontexts of a $C^*$-algebra strong Morita context, the situation is more complicated. Suppose $[c]$ is now also a holomorphic bundle, and $A$ is a subalgebra of holomorphic cross-sections as defined in Section \ref{exexample}. We would hope to have the refinements:

\begin{multline} \label{eqn:tentativediagram} A \text{ is strongly Morita equivalent to its center }{Z}  \\ \overset{(b')}{\iff} [c] \text{ is a holomorphically }PU_n(\mathbb{C})\text{-trivial bundle} \\ \overset{(c')}{\iff} A \text{ is $ci$ isomorphic to }{Z} \otimes M_n(\mathbb{C})
    \end{multline}

Note that there do exist holomorphically $PU_n(\mathbb{C})$-nontrivial bundles over $\overline{R}$. The forward direction of $(c')$ is true, and we conjecture that the backwards direction is true, but nontrivial. See \cite[Sec.~5.2]{McCormickThesis} for a proof of the backwards direction for $n=2$ and $\overline{R}$ being an annulus. Additionally, the backwards direction of $(b')$ is true quite immediately, while the forwards direction would be expected to be more difficult to prove. Therefore, there are two roadblocks to finishing Equation \ref{eqn:tentativediagram}, which amount to recovering topological information from algebras.

In Section \ref{exexample}, we will describe a first attempt at understanding $(b')$: finding a topological obstruction for a given strong Morita equivalence of an $n$-homogeneous $C^*$-algebra and its center to descend to a strong Morita equivalence of holomorphic subalgebras. However, the topological obstruction does not lead to a complete invariant, as it does not take into account all possible operator equivalence bimodules. There remains much work to be done in classifying strong Morita equivalence bimodules for nonselfadjoint algebras; for a result classifying all singly generated strong Morita $A$-$A$ bimodules, where $A$ is a logmodular function algebra, see \cite[Thm.~7.4]{Blecher2002}.

We might also consider a weaker question: In what settings is the holomorphic section algebra $A$ $cb$ Morita equivalent to its center? For our algebras $A$, the answer is ``always,'' because $A$ is always $cb$ isomorphic to $\mathcal{Z} \otimes M_n(\mathbb{C})$. This observation led us to unpack the key properties that are used the in the proof of Theorem \ref{mainprop}.

In Section \ref{notation}, we introduce our terminology and include some simple examples; in Section \ref{results}, we prove our main theorem. 

Section \ref{exexample} is devoted to an illustrative example of the subtleties surrounding $cb$ and strong Morita equivalence, including answering a question from \cite{Blecher2000} on generalizing their Example 8.3 in Theorem \ref{NotACIContext}.  This section also includes an application of Theorem \ref{mainprop} in Proposition \ref{prop:ApplicationOfMainThm}.

\section{Preliminaries} \label{notation}
We first introduce notation and recall some of the basic definitions from \cite{Blecher2000}. Operator space terminology and background may be found in references such as \cite{Paulsen2002, Blecher2004d}.

Let $A$ and $B$ be unital operator algebras, and unital subalgebras of $C^*$-algebras $A_1$ and $B_1$, respectively. Let $X$ be a left $A$ operator module; that is, $X$ is an operator space, $A$ is an operator algebra, and the module multiplication map $a \cdot x \mapsto ax, a \in A, x \in X$ is a completely bounded bilinear map. We assume that all operator modules are \emph{essential} in the sense that $AX$ is dense in $X$. Similarly, we can define a right $A$ operator module $Y$. Suppose that, in addition to being an $A$-module, $X$ is a right $B$ operator module and is an (algebraic) $A$-$B$ bimodule; then we call $X$ an operator $A$-$B$ bimodule. Let $Y$ be a $B$-$A$ operator bimodule. We will let $(\cdot, \cdot): X \times Y \to A$ be an $B$-balanced, bilinear, completely bounded map (or pairing), and let $[\cdot, \cdot]: Y \times X \to B$ be an $A$-balanced, bilinear, completely bounded pairing. By definition, these two pairings will induce $cb$ linear maps on the Haagerup tensor products $X \otimes_h Y$ and $Y \otimes_h X$, respectively.
 
\begin{definition}[Def.~3.1 from \cite{Blecher2000}] The data $(A,B,X,Y,(\cdot,\cdot),[\cdot,\cdot])$ is a called a \textit{$cb$ Morita context} or a \textit{Morita context for $A$ and $B$} if $A$, $B$, $X$, $Y$, $(\cdot, \cdot)$, and $[\cdot, \cdot]$ are as above, and the following compatibility and norm conditions are satisfied:
\begin{itemize}
    \item $(x_1, y)x_2 = x_1[y,x_2]$ for all $x_1,x_2 \in X$, $y \in Y$;
    \item $[y_1,x]y_2=y_1(x,y_2)$ for all $y_1,y_2 \in Y$, $x \in X$;
    \item the induced map $(\cdot, \cdot ) : (X \otimes_h Y)/ker(\cdot,\cdot) \to A$ is a surjective complete quotient map with $cb$ inverse; and
    \item the induced map $[\cdot, \cdot ] : (Y \otimes_h X)/ker(\cdot,\cdot) \to B$ is a surjective complete quotient map with $cb$ inverse.
    \end{itemize}
    \end{definition}

By \cite[Lem.~2.8, Rmk.~2.10]{Blecher2000}, the bilinear pairing $[\cdot,\cdot]$ induces a completely bounded quotient map on \mbox{$Y \otimes_h X/ \text{ker}[\cdot,\cdot]$} iff there is a $cb$ \emph{lifting of $1_B$}, i.e.~for every $\epsilon >0$, there are $x^i \in X$, $y^i \in Y$ (dependent on $\epsilon$) satisfying $1_B=\sum_{i=1}^k [y^i,x^i]$ and satisfying $\|(y^1,\ldots, y^k) \|_{cb}< C+\epsilon$, $\| (x^1, \ldots, x^k)^t\|_{cb} < C+\epsilon$ for some constant $C>0$ not dependent on $\epsilon$. 
The infimum of possible constants $C$ is called the \textit{norm} of the lifting. In this paper, we will say we have an \textit{algebraic lift} or just a \textit{lift} if we want a particular representation of the identity $1_B=\sum_{i=1}^k [y^i,x^i]$ using a specific collection of $x^i \in X$, $y^i\in Y$, $1\leq i \leq k$. An algebraic lift is a lifting with $cb$ norm $\max \{ \|(y^1,\ldots, y^k) \|_{cb}, \| (x^1, \ldots, x^k)^t\|_{cb} \}$, i.e.~we could construct a $cb$-lifting by choosing for every $\epsilon>0$ the same $x^i \in X$, $y^i\in Y$.

Parallel statements may be made about $cb$ liftings of $1_A$ by reversing the roles of $X$ and $Y$. We now make two definitions that relate to the hypotheses in our main theorem.

\begin{definition} An algebraic lift of $1_B$, $1_B=\sum_{i=1}^k [y^i,x^i]$, is called \emph{symmetric} if $(x^i,y^j)^*\in A$ for all $1 \leq i,j \leq k$.
	\end{definition}

Here, we can think of the adjoint $*$ as being taken in the ambient $C^*$-algebra $A_1$. If the context $(A=A_1,B=B_1,X=X_1,Y=Y_1,(\cdot,\cdot),[\cdot,\cdot])$ were a strong Morita context of $C^*$-algebras, then $Y_1$ would be completely linearly isometric to $X_1^*$ and $X_1$ would be completely linearly isometric to $Y_1^*$. Making the latter identifications, a sufficient condition for algebraic lift of $1_{B_1}$ to be symmetric would be that ${{x^i}}^* \in Y_1$ and ${{y^i}}^* \in X_1$ for every $i$. 

\begin{definition} \label{def:CompatibleSymmetric} Suppose that there are lifts $1_B=\sum_{i=1}^k [y^i,x^i]$ and $1_A=\sum_{i=1}^{k'} ({x^i}',{y^i}')$ where the lift of $1_B$ is symmetric. Then the pair of lifts is called $A$-\emph{compatible symmetric} if $({x^i}',y^j)^* \in A$ and $(x^i,{y^j}')^* \in A$ for all $i,j$. 
	\end{definition}

As a consequence of Definition \ref{def:CompatibleSymmetric}, given lifts $1_B$ and $1_A$ that are $A$-compatible symmetric, then the unital $C^*$-algebra generated by $1_A$ and the elements $(x^i,y^j)$, $({x^i}',y^j)$, $(x_i,{y^j}')$, $(x^i,y^j)^*$, $({x^i}',y^j)^*$, $(x_i,{y^j}')^*$, $1 \leq i,j \leq k$, will be a subset of $A$.

\begin{example} \label{ex:mainexample} Let $A:=A(\mathbb{D})$ be the disk algebra, i.e.~the collection of continuous functions from the closed unit disk $\mathbb{D}$ in the complex plane into the complex numbers, $\mathbb{C}$. Let $B:=M_n(A(\mathbb{D}))$, $X:=R_n(A(\mathbb{D}))$, and $Y:=C_n(A(\mathbb{D}))$, where $M_n$, $R_n$, and $C_n$ are the complex $n \times n$ matrices, $1 \times n$ row vectors, and $n \times 1$ column vectors, respectively. The algebra $A(\mathbb{D})$ inherits a unital operator algebra structure from being a subset of $C(\mathbb{D})$, the algebra of continuous functions, and the spaces $B$, $X$, and $Y$ inherit the natural operator algebra or operator space structure from $A$. If we let the bilinear pairings $(\cdot,\cdot)$ and $[\cdot,\cdot]$ denote matrix multiplication, then we have both an algebraic and strong Morita context $(A,B,X,Y,(\cdot,\cdot),[\cdot,\cdot])$. One algebraic lift of $1_A$ is given by $x^1:=\begin{pmatrix} 1 & 0 & 0 & \cdots & 0 \end{pmatrix}=(e_1)^t$, $y^1:=\begin{pmatrix} 1 \\ 0 \\ 0 \\ \vdots \\ 0 \end{pmatrix}=e_1$. One algebraic lift of $1_B$ is given by the collections $\{y^i \}=\{e_1,e_2,e_3,\ldots, e_n \}$, $\{ x^i\}=\{(e_1)^t,(e_2)^t,(e_3)^t,\ldots, (e_n)^t \}$. The lift of $1_A$ is symmetric trivially, the lift of $1_B$ is symmetric because the adjoint acts trivially on the diagonal matrix units, which belong to $M_n(A(\mathbb{D}))$. The pair of lifts is compatible symmetric because any of the combinations $({x^i}',y^j)^*$ and $(x^i,{y^j}')^*$ result in either the constant function 0 or 1, and thus belong to $A(\mathbb{D})$. The norm of the given algebraic lift of $1_A$ is 1, and the norm of the algebraic lift of $1_B$ is also 1, so this algebraic lift also gives us a $cb$ lifting of norm 1, producing a strong Morita context.
    \end{example}

\begin{example} \label{ex:C*-inclusionexample}
Kodaka and Teruya \cite[Def.~2.1]{KodakaTeruya2018} develop the notion of \emph{strong Morita equivalence for an inclusion of $C^*$-algebras}, with an eye towards analyzing conditional expecations. We will restrict to the setting of unital $C^*$-algebras, and compare their notion to a unital subcontext of two unital $C^*$-algebras. We will translate some of their definition to the language of Morita contexts and lifts.

For a strong Morita equivalence of a unital inclusion of unital $C^*$-algebras $A \subset A_1$, $B\subset B_1$, they require the following. First, they require an $A_1-B_1$ equivalence bimodule $X_1$ with closed subspace $X$ that is an $A-B$ submodule. The equivalence bimodule $X_1$ naturally induces a $B_1-A_1$ operator bimodule $Y_1=X_1^*$ (i.e., $X$ equipped with conjugate actions of $A_1$ and $B_1$) with closed $B-A$ submodule $X^*$. An equivalence bimodule comes with associated sesquilinear maps $_{A_1}\langle \rangle, \langle \rangle_{B_1}$, and we may identify these with pairings, $_{A_1}\langle x, y\rangle = (x,y^*)$ and $\langle x, y\rangle_{B_1} = [x^*,y]$. Moreover, they require that the pairings restricted to $X$ and $X^*$ ($X^*$ and $X$) map onto $A$ (and $B$). Since $A$ and $B$ are $C^*$-algebras, this will force any algebraic lifts of $1_A$ and $1_B$ to be symmetric, and all pairs of lifts to be compatible symmetric. Lastly, they ask that $(X_1,X^*)=A_1$ and $[X_1^*,X]=B_1$, another generating condition. If one assumes that there does exist an algebraic lift of $1_A$ and of $1_B$ (such as when $X$ is of finite type), this last condition will automatically be fulfilled.
    \end{example}

For the readers who are familiar with the $C^*$-algebra setting, it may be relevant to note that a strong Morita equivalence between general operator algebras $A$ and $B$ as in \cite{Blecher2000} is \emph{not} equivalent to stable isomorphism $A \otimes \mathcal{K} \simeq B \otimes \mathcal{K}$ \cite[Ex.~8.2]{Blecher2000}, as it is for $\sigma$-unital $C^*$-algebras \cite{BrownGreenRieffel1977}. \cite{Elef2016} has developed the notion of strong $\Delta$-equivalence of operator algebras, and strong $\Delta$-equivalence is satisfied iff the algebras are stably isomorphic, and has the additional property that strong $\Delta$-equivalence implies strong Morita equivalence (when assuming a contractive approximate identity).

\begin{remark} Most of the hypotheses to Theorem \ref{mainprop} are conditions that are intrinsic to the Morita context $(A,B,X,Y,(\cdot,\cdot),[\cdot,\cdot])$. For example, this is true of conditions of having a pair of compatible symmetric lifts. However, the first hypothesis is that  $(A,B,X,Y,(\cdot,\cdot),[\cdot,\cdot])$ is a subcontext of a $cb$ Morita context of $C^*$-algebras, which is an extrinsic condition. There are times in which one can lift a Morita context of operator algebras to a Morita context of ambient $C^*$-algebras, which would eliminate the need for this second hypothesis. In particular, \cite[Thm.~4]{Blecher1999a} accomplishes such a lifting when applied to unital operator algebras $A$ and $B$ contained in their unital $C^*$-envelopes $A_1$ and $B_1$ -- but only in the setting of a \textit{strong} Morita context of operator algebras, rather than a $cb$ context. 
	\end{remark}

\section{Results} \label{results}

In \cite[7.11]{Blecher2000}, the authors show that if $(A,B,X,Y,(\cdot,\cdot), [\cdot,\cdot])$ is a $cb$ Morita context where $A$ is a unital $C^*$-algebra and $B$ is a unital operator algebra, then $B$ is similar to a $C^*$-algebra that is strongly Morita equivalent to $A$. The main idea of the proof of Theorem \ref{mainprop} is to follow this outline as much as reasonably possible, given that our $A$ is not a $C^*$-algebra, and to use of the condition of having a pair of compatible symmetric lifts as a replacement for a $C^*$-algebra at certain points. 

We will start by stating the following lemma, which has been known in the literature.

\begin{lemma} \label{lem:idempotent-from-lifts} Let $1_B=\sum_{i=1}^k [y^i,x^i]$ where $x^i \in X$, $y^i\in Y$, $1\leq i \leq k$ be an algebraic lift of $1_B$. Then the operator $P=\begin{pmatrix} (x^i,y^j)_A \end{pmatrix}_{1 \leq i,j \leq k} \in M_k(A)$ is an idempotent.
    \end{lemma}

\begin{proof} Consider the computation
\begin{align*} \begin{pmatrix} (x^i,y^j)_A \end{pmatrix}_{1 \leq i,j \leq k} \cdot \begin{pmatrix} (x^i,y^j)_A \end{pmatrix}_{1 \leq i,j \leq k} & =  \begin{pmatrix} \sum_{j} (x^i,y^j)(x^j,y^l) \end{pmatrix}_{1 \leq i,l \leq k} \\
& = \begin{pmatrix} ( \sum_{j} (x^i,y^j)x^j,y^l) \end{pmatrix}_{1 \leq i,l \leq k} \\
 & = \begin{pmatrix} ( \sum_{j} x^i[y^j,x^j],y^l) \end{pmatrix}_{1 \leq i,l \leq k} \\
 &= \begin{pmatrix} (x^i,y^l)_A \end{pmatrix}_{1 \leq i,l \leq k}
	\end{align*}
    \end{proof}

\mainprop*

\begin{proof}[Proof of Theorem \ref{mainprop}] We start by following the general outline of \cite[7.11]{Blecher2000}, but keeping track of the subalgebra $A$.
Since $A$ is $cb$ Morita equivalent to $B$ and we have assumed the existence of a pair of compatible symmetric lifts, there is a symmetric lift of the identity \mbox{$1_B=\sum_{i=1}^k [y^i,x^i]$} where $y^i \in Y$, $x^i \in X$ for every $i=1,\ldots, k$, and which is compatible with a lift $1_A=\sum_{i=1}^{k'} ({x^i}',{y^i}')$. Note that $1_B=\sum_{i=1}^k [y^i,x^i]$ is also an algebraic lift of the identity for the Morita context for $A_1$ and $B_1$. Form an idempotent element $P \in M_k(A) \subset M_k(A_1)$ by setting
\[ P=\begin{pmatrix} (x^i,y^j)_A \end{pmatrix}_{1 \leq i,j \leq k} \] as in Lemma \ref{lem:idempotent-from-lifts}.

Since $M_k(A_1)$ is a $C^*$-algebra, by \cite[Thm.~26]{Kaplansky1968} there exists a self-adjoint projection $Q \in M_k(A_1)$ so that $PM_kA_1=QM_kA_1$ and $PQ=Q$, $QP=P$. In particular, one may take $Q=PP^*[1+(P-P^*)(P^*-P)]$.  Our assumption that the lift of $1_B$ is symmetric implies $Q \in M_k(A)$ as well. Further, by \cite[Thm.~15]{Kaplansky1968}, the map 
\[ \varphi: Q M_k(A_1)Q \to P M_k(A_1)P \]
given by $\varphi(x)= xP$ is a surjective, unital ring isomorphism, with $\| \varphi \|_{cb}=\|P\|$. Also $\varphi^{-1} : P M_k(A_1)P \to QM_k(A_1)Q$ is given by $\varphi^{-1}(x)=xQ$, and $\|\varphi^{-1}\|_{cb}=\|Q\|$.

We will next show that $B_1$ is unitally $cb$ isomorphic to $QM_k(A_1)Q$ via a map \\ \mbox{$f^{-1} \circ \varphi : QM_k(A_1)Q \to B_1$.} Define

\begin{align*} f : & B_1 \to PM_k(A_1)P \\ 
					& b \mapsto \left( (x^i,by^j) \right)_{1 \leq i,j \leq k}\end{align*}
The function $f$ is well-defined since clearly $f(b) \in M_k(A_1)$ and
\begin{align*} Pf(b)P = \left( (x^i,y^j) \right)_{ij} \left( (x^i,by^j) \right)_{ij} \left( (x^i,y^j) \right)_{ij} & = \left( \sum_j (x^i,y^j)(x^j,by^l) \right)_{il} \left( (x^i,y^j) \right)_{ij} \\
	& = \left( (\sum_j(x^i,y^j)x^j,by^l) \right)_{il} \left( (x^i,y^j) \right)_{ij} \\
	&= \left( (x^i,by^l) \right)_{il}  \left( (x^i,y^j) \right)_{ij}\\
	&= \left( \sum_j (x^ib,y^j)(x^j,y^l) \right)_{il} \\
	&= \left( (x^ib,y^l) \right)_{il}
	\end{align*}

We have that $f$ is surjective since if $\left( a_{ij} \right) \in M_k A_1$, then we can rewrite $P(a_{ij})P$ as:
\begin{align*} \left( (x^i,y^j) \right) \left( a_{jl} \right) \left( (x^l,y^m) \right) = \left( \sum_j (x^i,y^ja_{jl}) \right) \left( (x^l, y^m) \right) & = \left( \sum_l \sum_j (x^i,y^la_{jl})(x^l,y^m) \right) \\
	& = \left( \sum_k \sum_j (x^i y^ja_{jl}(x^l,y^m)) \right) \\
	& = \left( \sum_k \sum_j (x^i, [y^ja_{jl},x^l]y^m \right) \\
	& = \left( (x^i, (\sum_k \sum_j [y^ja_{jl},x^l])y^m) \right)
	\end{align*}
and $\sum_k \sum_j [y^ja_{jl},x^l] \in B_1$. Furthermore, note that this computation demonstrates that if $a_{ij} \in A$ for all $i,j$, then $P \left( a_{ij} \right) P \in M_k(A)$.

Since $(\cdot,\cdot)_{A_1}$ is completely bounded with, say, $cb$ norm $C$, and the $B_1$-module multiplication on $Y_1$ is completely bounded with $cb$ norm $K$ we have
\[ \left\| \left( (x^i,by^j) \right)_{i,j} \right \| \leq C \left\| \begin{pmatrix} x^1  \\ x^2 \\ \vdots \\ x^k  \end{pmatrix}
 \right\| \left\| \begin{pmatrix} y^1 & y^2 & \cdots & y^k  \end{pmatrix} \right\| {K} \|b\|  \]
Similarly, the map $f^{-1}$ is completely bounded. 

Since $QM_k(A_1)Q$ is a $C^*$-algebra, we have that there is a faithful $*$-representation $\rho : B_1 \to B(H_\rho)$ and an invertible operator $S \in B(H_\rho)$ so that $S\rho(\varphi^{-1}(f(\cdot)))S^{-1}: QM_k(A_1)Q \to S \rho(B_1) S^{-1}$ is a $*$-isomorphism \cite[Lem.~2.5]{Clouatre2015}. Further, \begin{equation} \label{Snormestimate} \|S\|^2=\|S^{-1}\|^2= \|\varphi^{-1} \circ f\|_{cb} \leq \|Q\| CK \|(x^1,x^2,\ldots, x^k)\|_{cb} \|(y^1,y^2,\ldots, y^k)^t\|_{cb}. \end{equation}
	
Before we saw that the image of $B$ in $PM_k(A_1)P$ is $PM_k(A)P$, i.e., $f(B)=PM_k(A)P$; and so, $B=f^{-1}(\varphi(QM_k(A)Q))$. Therefore, to show that $A$ is strongly Morita equivalent to $S\rho(B)S^{-1}$, it is enough to show that $A$ is strongly Morita equivalent to $QM_k(A)Q$. However, we first show that $S\rho(B_1)S^{-1}$ is Morita equivalent to $A_1$ via a two-step strong Morita equivalence:
\[ A_1 \sim M_k(A_1) \sim Q M_k(A_1) Q \simeq S\rho(B_1)S^{-1} \] 

The first strong Morita equivalence is clear, see \cite[3.4]{Blecher2000}. For the second, one would only need to check that $M_k(A_1)QM_k(A_1)=M_k(A_1)PM_k(A_1)$ is a dense ideal in $M_k(A_1)$ \cite[Lem.~1.1]{Brown1977b}.

Given $(a_{ij}), (b_{ij}) \in M_k(A_1)$, we have that
\[ (b_{ij}) \left( (x^i,y^i) \right) (a_{ij})  = \left( \sum_i \sum_j b_{li}(x^i,y^j)a_{jm}  \right)_{lm} = \left( \sum_i \sum_j (b_{li}x^i,y^ja_{jm})  \right)_{lm} \]
One observes that any $x \in X_1$ is an $A_1$-linear combination of the $x^i$, since $x = x \cdot 1_B = \sum_{j=1}^k x [y^j, x^j] = \sum_{j=1}^k (x,y^j)x^j$. Similarly, any $y \in Y_1$ is an $A_1$-linear combination of the $y^j$. Since one may get any $x \in X$ from $\sum_i b_{li}x^i$, and any $y \in Y$ from $\sum_j y^j a_{jm}$, and the pairing $(\cdot, \cdot)$ maps onto $A_1$, we have that by taking linear combinations, $M_k(A_1)PM_k(A_1)=M_k(A_1)$.

Finally, we will prove that by restricting the above Morita equivalences, we have strong Morita equivalences:
\[ A \sim M_k(A) \sim QM_k(A)Q \]

The first strong Morita equivalence is done in \cite[3.4]{Blecher2000}, so we concentrate on the second. As before, note that $Q \in M_k(A)$. We will check that the multiplication maps $QM_k(A) \otimes M_k(A)Q \to QM_k(A)Q$ and $M_k(A)Q \otimes Q M_k(A) \to M_k(A)$ are complete quotient maps. It is sufficient to (a) show the second map is onto, (b) find a norm 1 lifting of $1_{QM_k(A)Q}$, and (c) demonstrate the hypotheses of \cite[Lem.~2.8]{Blecher2000} are satisfied for the second quotient map.

For part (a), we follow an argument analogous to the above. We have that $M_k(A)QM_k(A)=M_k(A)PM_k(A)=M_k(A)$ as well; indeed, one only needs to replace $A$ for $A_1$, $B$ for $B_1$, $X$ for $X_1$ and $Y$ for $Y_1$ since the pairings $(\cdot,\cdot)$ and $[\cdot,\cdot]$ restrict to onto maps.

For part (b), set $X:=M_k(A)Q$ and $Y:=QM_k(A)$. Let $e_i$ be the matrix in $M_k(A)$ with the $(1,i)$ entry being $1_A$ and the rest of the entries 0, and set $x_i':=e_i^t Q$ and $y_i':=Q e_i$. Then $\|(y_1',y_2',\ldots, y_k')\|_{cb}=1$, $\|(x_1',x_2', \ldots, x_k')^t\|_{cb}=1$, and $\sum y_i'x_i'=Q=1_{QM_k(A)Q}$ which gives a norm 1 lifting of $1_{QM_k(A)Q}$. 
We had that $1_B= \sum_{i=1}^k [y^i,x^i]$ and $1_A= \sum_{i=1}^{k'} ({x^i}',{y^i}')$ are symmetric compatible lifts and so $(x^i,y^j)^*, ({x^i}',y^j)^*, (x_i,{y^j}')^* \in A$. Thus the unital $C^*$-algebra $C$ generated by $(x^i,y^j), ({x^i}',y^j), (x_i,{y^j}'), (x^i,y^j)^*, ({x^i}',y^j)^*, (x_i,{y^j}')^*$, and $1_A$ is a subset of $A$. We can see that $M_k(C)QM_k(C)$ contains $1_{M_k(C)}=1_{M_k(A)}$ since
\[ 1_C = 1_A = \sum_{i=1}^{k'} ({x^i}',{y^i}')
	\]
and any ${x^i}'\in X$ is an $A$-linear combination of the $x^i$ (and similarly for the ${y^i}'$). Thus $Q$ is a full projection in $C$. We may then follow the argument of \cite[5.11(iii)]{Blecher2000} to see that for every $\epsilon >0$, there are elements $F_n \in M_k(C)Q$, $1\leq n \leq k'$, with $\|(F_1,F_2,\ldots,F_{k'})\|_{cb}=1=\|(F_1^*,F_2^*,\ldots, F_{k'}^*)^t\|$ and \[1_{M_k(A)}=1_{M_k(C)}=\sum_{n=1}^{k'} F_n Q F_n^*\] Because $F_n,F_n^* \in M_k(A)$, \cite[Lem.~2.8]{Blecher2000} implies that $M_k(A)Q \otimes QM_k(A) \to M_k(A)$ is a complete quotient map.
	\end{proof}

 

\section{An extended example} \label{exexample}


Fix $R$ to be an open Riemann surface with smooth boundary $\partial R$ consisting of finitely many smooth boundary components, and such that $\overline{R}=R \cup \partial R$. Consider the classical function algebra $\mathcal{A}_1:= C(\overline{R})$, that is, the collection of continuous functions defined on $\overline{R}$ with values in $\mathbb{C}$, and $\mathcal{A}:=A(\overline{R})$, the subset of $C(\overline{R})$ consisting of continuous functions which are holomorphic on $R$.

Let $\widetilde{D}$ be the universal covering space of $\overline{R}$, and let $D$ be the universal covering space of $R$, realized as a dense subset of $\widetilde{D}$. Fix a representation $\rho : \pi_1(R) \to U_n(\mathbb{C})$, where $U_n(\mathbb{C})$ denotes the $n \times n$ unitary matrices over $\mathbb{C}$. Let $A_1$ be the algebra of bounded continuous functions $f: \widetilde{D} \to \mathbb{C}$ where $f(\mathfrak{z}\cdot g)=f(\mathfrak{z})$ for all $\mathfrak{z} \in \widetilde{D}$, $g \in \pi_1(R)$. Let $\mathcal{A}$ be the algebra of bounded continuous holomorphic functions $f: \widetilde{D} \to \mathbb{C}$, so that each $f \in \mathcal{A}$ is invariant under the action of $\pi_1(R)$: i.e.~$f(\mathfrak{z}\cdot g)=f(\mathfrak{z})$ for all $\mathfrak{z} \in \widetilde{D}$, $g \in \pi_1(R)$. Let $\mathcal{B}_1$ be the algebra of $n \times n$ matrices $F=(f_{ij})_{ij}$ such that each entry $f_{ij}$ is a bounded continuous function on $\widetilde{D}$, and such that $F$ is equivariant under the conjugation action of $\rho(\pi_1(R))$: i.e.~$F(\mathfrak{z}\cdot g)=\rho(g)\inv F(\mathfrak{z}) \rho(g)$ for all $\mathfrak{z} \in \widetilde{D}$, $g \in \pi_1(R)$. Let $\mathcal{B} \subseteq \mathcal{B}_1$ be the subalgebra of $F$ so that the entries $f_{ij}$ are bounded continuous holomorphic functions. Define $X_1$ to be the set of $1 \times n$ vectors $H=(f_i)_i$ of bounded continuous functions $f_i$ where the vector $H$ is equivariant under the action of right multiplication by $\rho(\pi_1(R))$, i.e. $H(\mathfrak{z} \cdot g)=H(\mathfrak{z})\rho(g)$ for all $g \in \pi_1(R)$. Let $X$ be the corresponding subspace where the $f_i$ are continuous holomorphic functions. Let $Y_1$ be the set of $n\times 1$ vectors $H'=(f_i)_i^t$ of bounded continuous holomorphic functions $f_i$ on $\widetilde{D}$ such that $H'(\mathfrak{z} \cdot g)=\rho(g)^{-1}H'(\mathfrak{z})$ for all $g \in \pi_1(R)$, and let $Y$ be the subspace where the $f_i$ are continuous holomorphic. 

The spaces $\mathcal{A}$, $\mathcal{B}$, $X$, and $Y$ all have natural operator space structures: Any element of one of these spaces is a matrix of continuous holomorphic functions. It is enough to specify that the norm of an entry of a matrix will be the supremum norm on this function. The norm of a matrix concomitant will be finite because the concomitant condition guarantees that $\| \cdot \|$ descends to a subharmonic function on $\overline{R}$. Once we have defined the norm on any of $\mathcal{A}$, $\mathcal{B}$, $X$, $Y$, there is a canonical way to obtain a norm on $M_n(\mathcal{A})$, $M_n(\mathcal{B})$, $M_n(X)$, $M_n(Y)$. 
 
At certain points in the proof, it is useful to identify the algebra $\mathcal{B}$ with the algebra of continuous holomorphic sections of a flat matrix bundle over $\overline{R}$. To make this identification, we first note the correspondence between the conjugation action of a unitary matrix $\rho(g)$ and the action of a $*$-algebra automorphism in $PU_n(\mathbb{C})$. Then apply, for example, \cite[4.8.1]{Husemoller1994}, using a similar argument to that of \cite[Sec.~2]{GriesenauerMuhlySolel2018}. In particular, a concomitant $F: \widetilde{D} \to M_n(\mathbb{C})$ is identified with a section $\sigma: \overline{R} \to E$ of $\mathfrak{E}_\rho(\overline{R})$, where $E =\widetilde{D} \times_{Ad(\rho)} M_n(\mathbb{C})$. 

The (prospective) $cb$ Morita context $(\mathcal{A}_1,\mathcal{B}_1, X_1,Y_1, (\cdot,\cdot)_1, [\cdot,\cdot]_1)$ and its $cb$-subcontext $(\mathcal{A},\mathcal{B}, X,Y, (\cdot,\cdot), [\cdot,\cdot])$ are described as follows. The bilinear pairings $(\cdot,\cdot)_1: X_1 \times Y_1 \to \mathcal{A}$ and $[\cdot,\cdot]_1 : Y_1 \times X_1 \to \mathcal{B}$ are defined by pointwise matrix multiplication of sections, and the pairings $(\cdot,\cdot)$ and $[\cdot,\cdot]$ are restrictions of the first two pairings, respectively. These two multiplications of $x \in X_1$ and $y \in Y_1$ have the correct ranges because if $a \in \pi_1(R)$, then $x(\mathfrak{z} \cdot a) y(\mathfrak{z} \cdot a)= x(\mathfrak{z}) \rho(a) \rho^{-1}(a) y(\mathfrak{z}) = x(\mathfrak{z}) y(\mathfrak{z})$ and $y(\mathfrak{z} \cdot a) x(\mathfrak{z} \cdot a) = \rho^{-1}(a) y(\mathfrak{z}) x(\mathfrak{z}) \rho(a)$. 

We will show that the contexts $(\mathcal{A}_1,\mathcal{B}_1, X_1,Y_1, (\cdot,\cdot)_1, [\cdot,\cdot]_1)$ and $(\mathcal{A},\mathcal{B}, X,Y, (\cdot,\cdot), [\cdot,\cdot])$, are algebraic Morita contexts and $cb$-Morita contexts. We will also show $(\mathcal{A},\mathcal{B}, X,Y, (\cdot,\cdot), [\cdot,\cdot])$ is not a strong Morita context under additional topological assumptions on $\mathcal{A}, \mathcal{B}, X,Y$.

\begin{lemma} Let $\rho : \pi_1(R) \to U_n(\mathbb{C})$ be any representation. Then the algebra $\mathcal{B}_1$ is $C^*$-isomorphic to $C(\overline{R}) \otimes M_n(\mathbb{C})$.
    \end{lemma}
    \begin{proof} As we remarked previously, we may think of elements in $\mathcal{B}_1$ as sections of a flat matrix $PU_n(\mathbb{C})$-bundle, with bundle space $\widetilde{D}\times_{Ad(\rho)}M_n(\mathbb{C})$. The bundle space of the underlying principal bundle is $\widetilde{D}\times_{Ad(\rho)}PU_n(\mathbb{C})$, and we identify $\widetilde{D}\times_{Ad(\rho)}PU_n(\mathbb{C}) \simeq (\widetilde{D}\times_\rho (U_n(\mathbb{C}))^*) \otimes (\widetilde{D}\times_{\rho^*} (U_n(\mathbb{C})))$. Since $\widetilde{D}\times_\rho (U_n(\mathbb{C}))^*$ is continuously $PU_n(\mathbb{C})$-trivial, there exists a continuous section $\sigma : \overline{R} \to \widetilde{D}\times_\rho (U_n(\mathbb{C}))^*$. Also, $\sigma^*: \overline{R} \to \widetilde{D}\times_{\rho^*} (U_n(\mathbb{C}))$ given by $\sigma^*(z)=\sigma(z)^*$ will be a continuous section of $\widetilde{D}\times_{\rho^*} (U_n(\mathbb{C}))$. Pick an orthonormal basis $v_1,v_2,\ldots,v_n$ of $\mathbb{C}^n$ (thought of as column vectors). Then $\sigma \cdot v_i$ is a continuous section of the associated vector bundle $\widetilde{D} \times_{\rho^*}\mathbb{C}^n$, and $v_i^*\cdot \sigma^*$ is a continuous section of the associated vector bundle $\widetilde{D} \times_\rho (\mathbb{C}^n)^*$, for all $i=1,\ldots, n$. Thus $ \sigma \cdot v_i \otimes v_j^* \cdot \sigma^*$ is a continuous section of the matrix bundle $\widetilde{D}\times_{Ad(\rho)}M_n(\mathbb{C})$ for each $i,j \in \{ 1,\ldots,n \}$. Also, the collection of sections $\{ \sigma \cdot v_i \otimes v_j^* \cdot \sigma^* \mid \; 1 \leq i,j \leq n\}$ are linearly independent over $\mathcal{A}_1$, and so form a basis for the algebra and $\mathcal{B}$-module $\mathcal{B}_1$.
    
    Let $E_{ij}$ denote the matrix unit with a 1 in the $i,j$ entry. We claim that $f: \mathcal{B}_1 \to \mathcal{A}_1 \otimes M_n(\mathbb{C})$, defined on the basis by
    \[ f(\sigma \cdot v_i \otimes v_j^* \cdot \sigma^*) = 1 \otimes E_{ij} \]
    is a $*$-isomorphism. It is well-defined,  $\mathcal{B}_1$-linear, and a $*$-linear map on the basis. To check that it is an algebra isomorphism, we observe that 
    \[ \left( \sigma \cdot v_i \otimes v_j^* \cdot \sigma^* \right) \left( \sigma \cdot v_k \otimes v_l^* \cdot \sigma^*\right) = 
    \begin{cases} \sigma \cdot v_i \otimes v_l^* \cdot \sigma^*& \text{ if } j=k \\
               0  & \text{ else }
        \end{cases}
        \] and 
    \[ (1 \otimes E_{ij})(1 \otimes E_{kl})= 
    \begin{cases} 1 \otimes E_{il} \text{ if } j=k \\
               0  & \text{ else }
        \end{cases}.
        \]
  Thus $f$ is a $*$-isomorphism of $C^*$-algebras.  
        \end{proof}

The $*$-isomorphism $f$ provides a way to construct a strong Morita context containing $\mathcal{A}_1$ and $\mathcal{B}_1$, which we will see in the next proposition, Lemma \ref{lem:strongMoritaContext}.

 However, the $C^*$-isomorphism $f$ does not preserve the holomorphic subalgebras in general, because $\sigma^*$ will not be a holomorphic section in general. So, the situation is more complicated for the algebras $\mathcal{A}$ and $\mathcal{B}$. 
	
\begin{lemma} Let $\rho : \pi_1(R) \to U_n(\mathbb{C})$ be any representation. Then $(\mathcal{A}_1,\mathcal{B}_1, X_1,Y_1, (\cdot,\cdot)_1, [\cdot,\cdot]_1)$ is a strong Morita context. \label{lem:strongMoritaContext}
	\end{lemma}
	\begin{proof} It's already clear that $\mathcal{A}_1$ is strongly Morita equivalent to $\mathcal{B}_1$, since $\mathcal{B}_1$ is $*$-isomorphic to a $C^*$-algebra that is strongly Morita equivalent to $\mathcal{A}_1=C(\overline{R})$. What we need to show is that $X$ and $Y$ implement this Morita context, which is also rather easy to do once we understand the $*$-isomorphism.
	
	The cross-sections of the bundle $\widetilde{D} \times_{\rho^*}\mathbb{C}^n$ can be identified with elements of $Y_1$, and the sections of $\widetilde{D} \times_\rho (\mathbb{C}^n)^*$ can be identified with elements of $X_1$. $X_1$ is a $\mathcal{B}_1$-$\mathcal{A}_1$ bimodule, and $Y_1$ is an $\mathcal{A}_1$-$\mathcal{B}_1$ bimodule. Multiplication of appropriately-sized matrices in $M_{n,m}(\mathbb{C})$ is completely contractive, and so matrix multiplication as a map from $(M_{m,n}(\mathbb{C})\otimes C(\widetilde{D})) \odot (M_{n,m}(\mathbb{C})\otimes C(\widetilde{D})$) to $M_{m,m}(\mathbb{C})\odot C(\widetilde{D})$ is also completely contractive. 
	
	Let $x \in X_1$, $y \in Y_1$, and let $\odot$ denote the algebraic tensor product. The matrix multiplication map, $m : X_1 \odot Y_1 \to \mathcal{A}_1$ (resp., $m: Y_1 \odot X_1 \to \mathcal{B}_1$), is an $\mathcal{A}_1$-$\mathcal{A}_1$- (resp.~$\mathcal{B}_1$-$\mathcal{B}_1$)-bimodule map. Matrix multiplication is also associative, i.e.~for $x,x' \in X_1$, $y,y' \in Y_1$, $x'm(y,x)=m(x',y)x$ and $y'm(x,y)=m(y',x)y$.
	
	Let $x'_i:=v_i^*\cdot \sigma^* \in X_1$ and $y'_i:=\sigma \cdot v_i \in Y_1$. Then we have:
	\begin{equation*}
(x'_i(\mathfrak{z}), y'_i(\mathfrak{z})) = 1_\mathcal{A} \text{, for all }i
	\end{equation*}
	and
\begin{equation*}
\sum_{n=1}^n [ y'_i(\mathfrak{z}), x'_i(\mathfrak{z})] =  1_\mathcal{B}.\end{equation*}

Thus $1_\mathcal{A}$ is the in range of $(\cdot,\cdot)_\mathcal{A}$ and $1_\mathcal{B}$ is in the range of $[\cdot,\cdot]_\mathcal{B}$, and the pairings are onto.

By construction, $\|[x_1,x_2,\ldots,x_n]\|, \|[y_1,y_2,\ldots,y_n]^t\|=1$, which means that by \cite[Lem.~2.8]{Blecher2000}, the associated quotient maps to $(\cdot ,\cdot )_1$ and $[ \cdot,\cdot ]_1$ are complete quotient maps.
		\end{proof}

Thus we have that $\mathcal{A}_1$ is strongly Morita equivalent to $\mathcal{B}_1$.



\begin{lemma} \label{PairingsOnto} The pairings $(\cdot, \cdot)_\mathcal{A}$ and $[\cdot, \cdot]_\mathcal{B}$ are surjective onto the algebras $\mathcal{A}$ and $\mathcal{B}$, respectively.
	\end{lemma} 
\begin{proof}
In \cite{McCormick2019}, it is shown that $\mathcal{A}$ and $\mathcal{B}$ are nonempty, and furthermore: for any $\mathfrak{z}_0 \in \widetilde{D}$ and $(a_{ij})\in M_n(\mathbb{C})$, there is an $F \in \mathcal{B}$ with $F(\mathfrak{z}_0)=(a_{ij})$. 

We now construct elements $y_i$ in $Y$ so that the $n \times n$ matrix $F$ formed by the columns $[y_1, y_2, \ldots, y_n]$ has the property that (i) $F(\mathfrak{z})$ is invertible for any $\mathfrak{z} \in \tilde{D}$, and (ii) $F^{-1}$ is a matrix whose entries are continuous holomorphic functions. The representation $\rho : \pi_1(R) \to U_n(\mathbb{C})$ defines a flat holomorphic vector bundle over $R$.  We may think of $\overline{R}$ as embedded in the open Riemann surface $\breve{R}$ and extend our vector bundle over $\overline{R}$ to a vector bundle over $\breve{R}$ as in \cite[pg.~306]{Widom1971}. By standard results \cite[30.1]{Forster2012}, there are $n$ linearly independent holomorphic sections, $\sigma_1, \ldots, \sigma_n$, of the extended vector bundle over the open Riemann surface $\breve{R}$. Restrict each section $\sigma$ to $\overline{R}$; these will be sections of the original bundle over $\overline{R}$. Each section $\sigma_i$ of a flat holomorphic vector bundle over $\overline{R}$ can be identified with a $\pi_1(R)$-equivariant function $y_i \in Y$. Define $F$ to be the matrix $[y_1, y_2, \ldots y_n]$, and let $x_i$ be the $i$th row of $F^{-1}$. The entries of $F^{-1}$ are certainly continuous holomorphic functions on $\tilde{D}$ since by Cramer's rule we just need to check that $det(F(\mathfrak{z}))$ is never zero. Then $x_i$ is a vector of continuous holomorphic functions. Each $x_i$ and $y_i$ also satisfies the appropriate equivariance relation.
That is, we have:
\begin{equation*} \begin{split}
y_i(\mathfrak{z} \cdot g)& = \rho(g \inv) y_i(\mathfrak{z})\\
x_i(\mathfrak{z} \cdot g)& =  x_i(\mathfrak{z}) \rho(g).
	\end{split} \end{equation*}
By definition, $[y_1, y_2, \ldots, y_n] [x_1, x_2, \ldots x_n]^t=I_n = [x_1, x_2, \ldots x_n]^t [y_1 y_2 \ldots y_n]$. This means that
\begin{equation*}
(x_i(\mathfrak{z}), y_i(\mathfrak{z})) = 1_\mathcal{A} \text{, for all }i
	\end{equation*}
	and
\begin{equation*}
\sum_{n=1}^n [ y_i(\mathfrak{z}), x_i(\mathfrak{z})] =  1_\mathcal{B}.
	\end{equation*}   
Thus $1_\mathcal{A}$ is the in range of $(\cdot,\cdot)_\mathcal{A}$ and $1_\mathcal{B}$ is in the range of $[\cdot,\cdot]_\mathcal{B}$, and the pairings are onto.
\end{proof}

\begin{thm} $(\mathcal{A},\mathcal{B}, X,Y, (\cdot,\cdot)_{\mathcal{A}}, [\cdot,\cdot]_{\mathcal{B}})$ is a $cb$ Morita context. \label{IsACBContext}
	\end{thm}
	\begin{proof} Note that in \cite{McCormick2019} it is shown that $\mathcal{B}$ is $cb$ isomorphic to $\mathcal{A} \otimes M_n(\mathbb{C})$. Since any two algebras that are $cb$ isomorphic are in particular $cb$ Morita equivalent, what we really need to demonstrate is that $X$ and $Y$ implement the aforementioned Morita equivalence.

	The algebras $\mathcal{A}$ and $\mathcal{B}$ are unital and $1_\mathcal{A}$ and $1_\mathcal{B}$ act as identity on $X$ and $Y$. Therefore, $X$ is an essential $\mathcal{A}$-$\mathcal{B}$-bimodule, and $Y$ is an essential $\mathcal{B}$-$\mathcal{A}$-bimodule. We have shown all requirements of a $cb$ Morita context except the last two, namely that the induced linear maps $(\cdot)_\mathcal{A}: X \otimes_{h\mathcal{A}} Y/ker(\cdot)_\mathcal{A} \to \mathcal{A}$ and $[\cdot]_\mathcal{B}:Y \otimes_{h\mathcal{A}} X/ ker[\cdot]_\mathcal{B} \to \mathcal{B}$ are completely bounded with completely bounded inverses. 
	
	It is straightforward to see that the pairings are completely bounded. For, multiplication of appropriately-sized matrices in $M_{n,m}(\mathbb{C})$ is completely contractive \cite[page]{Blecher2004d}, and so matrix multiplication as a map from $(M_{m,n}(\mathbb{C})\otimes C(\widetilde{D})) \odot (M_{n,m}(\mathbb{C})\otimes C(\widetilde{D})$) to $M_{m,m}(\mathbb{C})\odot C(\widetilde{D})$ is also completely contractive.
	
	Above in Lemma \ref{PairingsOnto} we gave a $cb$ lifting $\sum y_i \otimes x_i$ of $1_\mathcal{B}$. That is, $(y_i\otimes x_i)_\mathcal{B}=1_\mathcal{B}$ and the $cb$-norm of $[\cdot,\cdot]_\mathcal{B}^{-1}: \mathcal{B} \to Y \odot X/ker[\cdot,\cdot]_\mathcal{B}$ is bounded above by $\|[y_i] \| \|[x_i]^t \|_{cb}$. Similarly, we have a $cb$ lifting of $1_\mathcal{A}$. 
    \end{proof}

\begin{prop} The algebra $\mathcal{B}$ is similar to an operator algebra $\mathcal{B}'$ which is strongly Morita equivalent to $\mathcal{A}$. \label{prop:ApplicationOfMainThm}
    \end{prop}

\begin{proof} The algebraic lifts of $1_\mathcal{A}$ and $1_\mathcal{B}$ mentioned in Lemma \ref{PairingsOnto} are compatible symmetric by construction, because of the equation 
\[ [y_1, y_2, \ldots, y_n] [x_1, x_2, \ldots x_n]^t=I_n = [x_1, x_2, \ldots x_n]^t [y_1 y_2 \ldots y_n] \] We saw in Theorem \ref{IsACBContext} that the context $(\mathcal{A},\mathcal{B}, X,Y, (\cdot,\cdot)_{\mathcal{A}}, [\cdot,\cdot]_{\mathcal{B}})$ is a unital $cb$ subcontext of $(\mathcal{A}_1,\mathcal{B}_1, X_1,Y_1, (\cdot,\cdot)_1, [\cdot,\cdot]_1)$. Therefore, the conclusion follows from Theorem \ref{mainprop}.
    \end{proof}

We will next demonstrate that in Theorem \ref{IsACBContext}, we cannot replace ``completely bounded'' with ``completely isometric'', i.e.~we cannot expect a Theorem \ref{lem:strongMoritaContext} for any subcontext, independent of $\rho$. However, it is first necessary to make a few remarks. 

The algebras $\mathcal{A}$ and $\mathcal{B}$ are defined with respect to a representation $\rho : \pi_1(R) \to U_n(\mathbb{C})$ (or more accurately, by the associated representation $Ad(\rho)$). Suppose we realize concretely $\pi_1(R)=\pi_1(R,s)$ based at a point $s \in R$. Then we may think of $\rho$ as a map taking equivalence classes $[a]$ of loops $a$ based at $s$ to $U_n(\mathbb{C})$. Representing $\rho$ in such a way is well-defined on equivalence classes, for any two loops with $[a]=[b]$ induce the same transformation of the covering space $\widetilde{D}$ and thus produce the same underlying principle bundle. Thus once fixing $s$ we may without ambiguity write $\rho([a])=\rho(a)$ for a loop $a$ based at $s$. 

If $\rho : \pi_1(R) \to U_n(\mathbb{C})$, and $U \in U_n(\mathbb{C})$, then $\rho U : \pi_1(R) \to U_n(\mathbb{C})$ given by $\rho U(a) := \rho(a) U $ determines the same algebras $\mathcal{B}$ and $\mathcal{B}_1$, but different modules $X,X_1, Y, Y_1$. So, suppose $\mathcal{B} = C(\overline{R}) \otimes M_n(\mathcal{C})$ and $\mathcal{B}_1 = A(\overline{R}) \otimes M_n(\mathcal{C})$. In this case, we can think of $\mathcal{B}$ or $\mathcal{B}_1$ as being determined by the trivial representation $\rho: \pi_1(R) \to U_n(\mathbb{C}) $, $\rho(\pi_1(R))=I_n$ or by a representation $\rho'=\rho U$ for some $U \in U_n(\mathbb{C})$ (since the algebras $\mathcal{B}$ and $\mathcal{B}_1$ are really determined by the representation $Ad(\rho)$). However, $\rho$ and $\rho'$ will give rise to different modules $X,X_1,Y,Y_1$. The theorem below will have the consequence that even though $\mathcal{A}$ is strongly Morita equivalent to $\mathcal{A}\otimes M_n(\mathbb{C})$, $\mathcal{A}$ will not be strongly Morita equivalent to $\mathcal{A} \otimes M_n(\mathbb{C})$ via the ($cb$) Morita context  $(\mathcal{A},\mathcal{B},X({\rho '}),Y({\rho '}),(\cdot,\cdot),[\cdot,\cdot])$ when $U$ does not have an eigenvalue of 1. In particular, Lemma \ref{lem:strongMoritaContext} does not hold for the subcontext $(\mathcal{A},\mathcal{B},X({\rho '}),Y({\rho '}),(\cdot,\cdot),[\cdot,\cdot])$ of $(\mathcal{A}_1,\mathcal{B}_1,X_1 ({\rho '}),Y_1 ({\rho '}),(\cdot,\cdot)_1,[\cdot,\cdot])_1$.

The proof of Thm.~\ref{NotACIContext} below is inspired by \cite[Ex.~8.3]{Blecher2000}, which is a special case of the theorem when $\overline{R}$ is an annulus.

\begin{thm} Let $X$, $Y$, $(\cdot,\cdot)$, and $[\cdot,\cdot]$ be chosen as above, and make the assumption that given $s \in R$, there is an element $a \in \pi_1(R,s)$ so that the operator $\rho(a)$ does not have an eigenvalue of 1. Then the context $(\mathcal{A},\mathcal{B},X,Y,(\cdot,\cdot),[\cdot,\cdot])$ is not a strong Morita context. \label{NotACIContext}
	\end{thm}
	\begin{proof}	
Suppose that $(\mathcal{A},\mathcal{B},X,Y,(\cdot,\cdot),[\cdot,\cdot])$ were a strong Morita context, and so given $\epsilon>0$ there existed $x_1,x_2,\ldots x_m \in X$ and $y_1,y_2,\ldots ,y_m \in Y$ such that $1_\mathcal{B} = \sum_{i=1}^m[y_i,x_i ]$, and where $\| (y_1,y_2,\ldots ,y_m)\|<1+\epsilon$ and $\|(x_1,\dots, x_m)^t \| < 1+\epsilon$. Let $G$ be the $n \times m$ matrix $G=(y_1,y_2,\ldots ,y_m)$ and let $F$ be the $m \times n$ matrix $(x_1,\dots , x_m)^t$. Then $G$ and $F$ are matrices of continuous functions on $\widetilde{D}$ that are analytic on the interior; $G(\mathfrak{z} \cdot g)=\rho(g\inv) G(\mathfrak{z})$ and $F(\mathfrak{z}\cdot g)=F(\mathfrak{z})\rho(g)$ for all $g \in \pi_1(\overline{R})$; $ \| (y_1,y_2,\ldots ,y_m)\| = \sup \{ \|G(\mathfrak{z})\| \mid \, \mathfrak{z} \in \widetilde{D} \}$ and $ \|(x_1,\dots , x_m)^t \| = \sup \{ \|F(\mathfrak{z})\| \mid \, \mathfrak{z} \in \widetilde{D} \}$; and $GF=I_n$. Furthermore, $0 \leq G G^* \leq (1+\epsilon)^2 I$ and $0 \leq F^* F \leq (1+\epsilon)^2 I$, so that $0 \leq (G-F^*)(G^*-F) = GG^* -2I+F^*F \leq 2 \epsilon (2+\epsilon ) I$. This shows, in particular, that  $\| G-F^*\|^2 \leq 2 \epsilon(2+\epsilon)$.

Let $\mathfrak{x}\in D$. The facts that $GF=I$ and $\|F\| \leq 1+\epsilon$ imply on the one hand that $\|G(\mathfrak{x})\| \geq (1+\epsilon)^{-1}$. We shall show, on the other hand, that $\|G(\mathfrak{x})\|$ is bounded above by an expression that goes to zero with $\epsilon$. Together, these are impossible, so $(\mathcal{A},\mathcal{B},X,Y,(\cdot,\cdot),[\cdot,\cdot])$ is not a Morita context.

Following \cite[Ex.~8.3]{Blecher2000}, the main idea of the proof is to estimate the quantity $\|G(\mathfrak{x})\|$. For this purpose, we rewrite $\|G(\mathfrak{x})-\rho(a)G(\mathfrak{x})\|=\| (I-\rho(a))G(\mathfrak{x}) \|$. Applying our technical assumption, $(I-\rho(a))^{-1}$ exists, so we may put $M:=\| (I-\rho(a))^{-1} \|$. Then  $\| G(\mathfrak{x}) \| \leq M \| G(\mathfrak{x})-\rho(a)G(\mathfrak{x}) \|$. The tricky part then comes from estimating $ \| G(\mathfrak{x})-\rho(a)G(\mathfrak{x}) \| $, which was done in \cite[Ex.~8.3]{Blecher2000} using a combination of Cauchy integral formula estimates and the triangle inequality. Our proof follows the same path, but requires more careful choices as we are working with an arbitrary bordered Riemann surface $R$ rather than the annulus.

Let $r_0 \in R$. Pick one representative $\mathfrak{x}_0 \in D$ such that $\pi(\mathfrak{x}_0)=r_0$, where $\pi : \tilde{D} \to \overline{R}$ is the universal covering map. Let $a$ be a loop based at $r_0$ where $\rho(a)$ does not have 1 as an eigenvalue. Identify $a$ with its action $a \in PSL_2(\mathbb{R})$ on $D$. Then the loop $\{a(t) \mid 0 \leq t \leq 1 \}$ lifts to a path $\{ \tilde{a}(t) \mid 0 \leq t \leq 1\} \subset D$ with $\tilde{a}(0)=\mathfrak{x_0}$ and $\tilde{a}(1)=a\cdot \mathfrak{x}_0$. Let $b(t):=\mathfrak{x}_0 (1-t) + a \cdot \mathfrak{x}_0 t$ be the straight line path between $\mathfrak{x}_0$ and $a \cdot \mathfrak{x}_0$.

We may find $r>0$ and $t_i \in [0,1]$ so that we may cover $\{ b(t) \}$ with $k+1 \in \mathbb{N}$ discs $D_r(x_i)$ $i=0, \ldots k$ of fixed radius $r>0$ centered at $x_i:=b(t_i)$ so that the following conditions hold:
	\begin{enumerate}[label=(\roman*)]
	\item We have $D_r(b(t)) \subset D$ for all $t$.
	\item $x_0:=\mathfrak{x}_0$, $x_k:=a \cdot \mathfrak{x}_0$.
	\item \label{secondcond} (a) $D_r(x_{i-1}) \cap D_r(x_i)$ contains a disk of radius $r/3$ centered at a point on $b$, $i=1, \ldots k$; (b) $D_r(x_i) \cap D_r(x_{i+1})$ contains a disk of radius $r/3$ centered at a point on $b$, $i=0, \ldots, k-1$.
	  \item Using \ref{secondcond}, we may pick points $x_i' \in D_r(x_i)\cap D_r(x_{i+1})$, $i=0, \ldots, k-1$ with $|x_i-x_i'|<r$ and $|x_i'-x_{i+_1}|<r$ and so $x_i, x_{i+1} \in D_r(x_i')$
		\end{enumerate}
All conditions are fulfilled if $r$ is less than the distance from the image of $b$ to $\partial D$ and if $k$ is the length of $b$ multiplied by $\dfrac{2}{3r}$ and the $x_i$ are chosen to be equally space along $b$.
	
Let $L=\inf \{ |w-x_i| \mid w \in \partial D_r(x_i'), i=0, \ldots, k-1 \}$. Then $L>0$ because there are finitely many $i$'s. 

We now begin a series of estimates:

\begin{multline*} \|G(x_0) - \rho(a) \cdot G(x_0) \| =  \| G(x_0)-G(a \cdot x_0) \| = \|G(x_0)-G(x_k)\| = \\
\leq \|G(x_0) - F^*(x_0') \| + \|F^*(x_0')-G(x_1)\| + \| G(x_1)-F^*(x_1')\| + \|F^*(x_1') - G(x_2) \| + \ldots \\
+ \| G(x_{x-1})-F^*(x_{k-1}') \| + \|F^*(x_{k-1}')-G(x_k) \|
	\end{multline*}

We will estimate each of the summands above. First, though, recall that there is an antiholomorphic Cauchy integral formula that reads as follows in the scalar case: Let $f$ be holomorphic function. Then 
 \[ \frac{1}{2 \pi i} \int \frac{\overline{f(w)}}{ (w -z)} dw = \overline{f(w_0)} \]
where $w$ runs over a circle centered at $w_0$ and $z$ is any fixed point inside the circle.

Consider the summand $\| G(x_i)-F^*(x_i')\|$:

\begin{multline*} G(x_i)-F^*(x_i') = \dfrac{1}{2\pi i} \int_{\partial D_r(x_i')} \dfrac{G(w)}{w-x_i'} \, dw - \dfrac{1}{2\pi i} \int_{\partial D_r(x_i')} \dfrac{F^*(w)}{w-x_i} \, dw
\\ = \dfrac{1}{2\pi i} \int_{\partial D_r(x_i')} \dfrac{G(w)-F^*(w)}{w-x_i}
	\end{multline*}
The first equality above comes from applying the Cauchy integral formula to $G$ (integrating around an `arbitrary curve' around $x_i$), and applying the antiholomorphic Cauchy integral formula to $F^*$ (integrating around a circle around $x_i'$, where the denominator of the integrand is the difference between $w$ and any point inside this circle).

Therefore, 
\begin{multline*} \| G(x_i)-F^*(x_i') \| \leq \dfrac{1}{2\pi} \sup_{w \in \partial D_r(x_i')} \dfrac{\| G(w)-F^*(w) \|}{\|w-x_i\|} \leq \dfrac{1}{2\pi L} \|G-F^*\| \leq \dfrac{\sqrt{2\epsilon(2+\epsilon)}}{2\pi L}
	\end{multline*}

Similarly, we may estimate $\|F^*(x_i')-G(x_{i+1})\|$:

\begin{multline*} F^*(x_i')-G(x_{i+1}) = \dfrac{1}{2\pi i} \int_{\partial D_r(x_i')} \dfrac{F^*(w)}{w-x_{i+1}} \, dw - \dfrac{1}{2\pi i} \int_{\partial D_r(x_i')} \dfrac{G(w)}{w-x_{i+1}} \, dw \\ = \dfrac{1}{2\pi i} \int_{\partial D_r(x_i')} \dfrac{F^*(w)-G(w)}{w-x_{i+1}} \, dw
	\end{multline*}
Where the first equality comes from applying the antiholomorphic Cauchy integral formula to $F^*$, and the Cauchy integral formula to $G$.

Therefore, we also have
\begin{equation*} \| F^*(x_i')-G(x_{i+1}) \| \leq \dfrac{\sqrt{2 \epsilon (2 + \epsilon)}}{2 \pi L}
	\end{equation*}

Thus we conclude

\begin{equation*} \|G(x_0)-G(x_k)\|
\leq (2k-1) \dfrac{\sqrt{2 \epsilon (2 + \epsilon)}}{2 \pi L}
	\end{equation*}

The numbers $k$ and $L$ come from properties of $R$, and thus are independent of $\epsilon$. We get, then, that $\|G(\mathfrak{x})\| \leq M \|G(\mathfrak{x})-\rho(a)G(\mathfrak{x})\| \leq M (2k-1) \dfrac{\sqrt{2 \epsilon (2 + \epsilon)}}{2 \pi L}$, which goes to $0$ with $\epsilon$. This contradicts our original estimate that $\|G(\mathfrak{x}) \| \geq \dfrac{1}{1+\epsilon}$, thus $(\mathcal{A},\mathcal{B},X,Y,(\cdot,\cdot), [\cdot,\cdot])$ is not a strong Morita context.

    \end{proof}

\noindent {\it Acknowledgements:} The author would like to thank Paul Muhly and Vern Paulsen, for helpful conversations and insights on their work; and Adam Dor-On, for some comments that improved the introduction.

\bibliographystyle{plain}
\bibliography{biblio12-27-19}

	\end{document}